\documentclass{article}

\usepackage{graphicx}
\usepackage{amsfonts,amsmath,amssymb,amsthm}
\usepackage{xcolor}

\newtheorem{theorem}{Theorem}
\newtheorem*{quotedtheorem}{Theorem}
\newtheorem{lemma}{Lemma}
\newtheorem{corollary}{Corollary}

\newtheorem{remark}{Remark}
\theoremstyle{definition}
\newtheorem{problem}{Problem}

\newcommand{\ex}{\operatorname{ex}}

\title{Tur\'an-Type Bounds for Graphs Containing Large $F$-Sparse Sets}
\author{Yupei Li
\thanks{University of South Carolina, Columbia, SC 29208, ({\tt yupei@email.sc.edu})}
 \and Linyuan Lu\thanks{University of South Carolina, Columbia, SC 29208,
({\tt lu@math.sc.edu}).
This author was supported in part by NSF grant DMS 2038080.}
}

\begin{document}

\maketitle

\begin{abstract}
We study Tur\'an-type extremal problems for graphs containing a large $F$-sparse vertex set, meaning a vertex set whose induced subgraph contains few copies of $F$. For integers $r>s\ge 1$, we prove that if a $K_{r+1}$-free graph $G$ on $n$ vertices contains a set $M$ of size $m\ge \lceil sn/r\rceil$ such that $G[M]$ is $K_{s+1}$-free, then
\[
   e(G)\le m(n-m)+t_s(m)+t_{r-s}(n-m).
\]
We characterize the equality cases as the complete $r$-partite graphs whose vertex classes split into two balanced groups of total sizes $m$ and $n-m$, consisting of $s$ and $r-s$ classes, respectively. We also prove a color-critical extension for forbidden graphs that embed into a join of two edge-critical graphs, together with an asymptotic extension for general $H$-free graphs in which the prescribed large vertex set spans few copies of a fixed graph $F$ with $\chi(F)<\chi(H)$.
\end{abstract}

\section{Introduction}

The Erd\H{o}s--Stone--Simonovits theorem is a cornerstone of extremal graph theory: it determines the asymptotic value of the Tur\'an number for every non-bipartite graph $H$. In the special case where the forbidden graph is the complete graph $K_{r+1}$, Tur\'an's theorem gives the exact extremal result: 
$\ex(n, K_{r+1})=t_r(n)$; 
the unique extremal graph is the Tur\'an graph $T_r(n)$, the complete $r$-partite graph on $n$ vertices whose vertex classes form an equitable partition.  
More generally, $T_r(n)$ is asymptotically extremal for every graph $H$ with chromatic number $\chi(H)=r+1$.

In this paper, we study how additional structural constraints on the host graph affect the corresponding extremal problem. We are particularly interested in situations where an $H$-free graph is required to contain a large prescribed vertex set and where this additional requirement forces the maximum number of edges to be smaller than the classical Tur\'an bound.

A first result in this direction was obtained by Balister, Bollob\'as, Riordan, and Schelp~\cite{BBRS}, who studied $C_{2k+1}$-free graphs under a maximum-degree constraint.

\begin{quotedtheorem}[\cite{BBRS}]
For sufficiently large $n$, if $G$ is an $n$-vertex $C_{2k+1}$-free graph with
\[
   n/2 \leq \Delta(G) \leq n-k-1,
\]
then
\[
   e(G)\leq \Delta(G)(n-\Delta(G)),
\]
with equality attained by the complete bipartite graph $K_{\Delta(G),n-\Delta(G)}$.
\end{quotedtheorem}

Huo and Yuan~\cite{HY} later generalized this result to all edge-critical graphs. Recall that a graph is \emph{edge-critical}, or \emph{color-critical}, if it contains an edge whose deletion reduces its chromatic number. In particular, every odd cycle is edge-critical.

\begin{quotedtheorem}[\cite{HY}]
Let $F$ be an edge-critical graph with $\chi(F)=r+1$. There exists a constant $s=s(F)<1$ such that, for sufficiently large $n$, if $G$ is an $n$-vertex $F$-free graph with
\[
   \frac{r-1}{r}n \leq \Delta(G) \leq n-\Theta(n^s),
\]
then
\[
   e(G)\leq \Delta(G)(n-\Delta(G))+t_{r-1}(\Delta(G)),
\]
with equality attained by the complete $r$-partite graph
\[
   \overline{K_{n-\Delta(G)}}+T_{r-1}(\Delta(G)).
\]
\end{quotedtheorem}
They also showed that one may take $s=0$ when $F$ is progressive-edge-critical. Here a graph $F$ is \emph{progressive-edge-critical} if $F$ is edge-critical and there exists a vertex $v$ such that $F-v$ is edge-critical and $\chi(F-v)=\chi(F)-1$.

For our exact extension, we use a related two-sided criticality condition. Call a graph $H$ \emph{double-edge-critical} if it contains two edges $e_1,e_2$ such that
\[
   \chi\bigl(H-\{e_1,e_2\}\bigr)=\chi(H)-2.
\]
Such edges are necessarily vertex-disjoint. Indeed, if $e_1$ and $e_2$ shared a vertex $v$, then $H-v$ would be a subgraph of $H-\{e_1,e_2\}$, and hence
\[
   \chi(H-v)\le \chi(H)-2.
\]
This contradicts the elementary inequality $\chi(H-v)\ge \chi(H)-1$, since deleting a single vertex can decrease the chromatic number by at most one.
Lemma~\ref{lem:double-critical-join} below shows that, whenever $r>s\ge 2$ and $r-s\ge 2$, every double-edge-critical graph $H$ with $\chi(H)=r+1$ embeds into the join $F_1+F$ of two edge-critical graphs satisfying
\[
   \chi(F_1)=r-s,
   \qquad
   \chi(F)=s+1.
\]
This structural decomposition allows us to extend the exact bound in Theorem~\ref{thm:1} to this family of forbidden graphs.

More recently, Behague, Chakraborti, and Liu~\cite{BCL} considered a different way of ``sabotaging'' a classical extremal theorem. They studied the maximum number of edges in a triangle-free graph that is required to contain a prescribed graph $\mathbb P$ as a subgraph; that is,
\[
\ex_{\mathbb P}(n,K_3)
=
\max\Bigl\{e(G): |V(G)|=n,\; G \text{ is } K_3\text{-free, and } \mathbb P\subseteq G\Bigr\}.
\]
For example, requiring $\Delta(G)\ge m$ is equivalent to forcing a copy of the star $K_{1,m}$ in $G$. When $m>\lceil n/2\rceil$, this star is not contained in the balanced complete bipartite graph $K_{\lfloor n/2\rfloor,\lceil n/2\rceil}$, so one naturally expects the Mantel bound to decrease.

This motivates the broader question of how the Tur\'an bound changes when one imposes an additional large structure on the host graph. In the present paper, the imposed structure is a large induced subgraph with bounded clique number. For $s\ge 1$, define
\[
   \alpha_s(G)
   =
   \max\bigl\{|M|: M\subseteq V(G) \text{ and } G[M] \text{ is } K_{s+1}\text{-free}\bigr\}.
\]
Thus $\alpha_1(G)=\alpha(G)$ is the usual independence number.

Our first main result gives the exact extremal number for $K_{r+1}$-free graphs with a prescribed large $K_{s+1}$-free vertex set.

\begin{theorem}\label{thm:1}
Let $r>s\ge 1$, and let $G$ be a $K_{r+1}$-free graph on $n$ vertices. Suppose that
\[
   \alpha_s(G)\ge m\ge \left\lceil \frac{sn}{r}\right\rceil.
\]
Then
\[
   e(G)\le m(n-m)+t_{r-s}(n-m)+t_s(m).
\]
Moreover, equality holds if and only if $G$ is a complete $r$-partite graph whose parts can be divided into two groups: $s$ parts with total size $m$, as equal as possible, and $r-s$ parts with total size $n-m$, as equal as possible.
\end{theorem}


Our next result gives a color-critical extension for double-edge-critical forbidden graphs.

\begin{theorem}\label{thm:2}
Let $r>s\ge 2$ with $r-s\ge 2$, and let $H$ be a double-edge-critical graph with $\chi(H)=r+1$. Fix edge-critical graphs $F_1$ and $F$ supplied by Lemma~\ref{lem:double-critical-join}; thus
\[
   \chi(F_1)=r-s,
   \qquad
   \chi(F)=s+1,
   \qquad
   H\subseteq F_1+F.
\]
Then there exists $n_0=n_0(H,F_1,F,r,s)$ such that the following holds for all $n\ge n_0$.

Let $G$ be an $H$-free graph on $n$ vertices. Suppose that $G$ contains a set $M$ of size $m$ such that $G[M]$ is $F$-free and
\[
  m\ge \left\lceil \frac{sn}{r}\right\rceil.
\]
Put $q=n-m$ and $J_0=F_1+K_1$. Then
\[
   e(G)\le m(n-m)+t_s(m)+\ex(q,J_0).
\]
Consequently, there exists $q_0=q_0(F_1,r,s)$ such that, whenever $n-m\ge q_0$,
\[
   e(G)\le m(n-m)+t_s(m)+t_{r-s}(n-m).
\]
In this latter range, equality holds if and only if $G$ is a complete $r$-partite graph whose parts can be divided into two groups: $s$ parts with total size $m$, as equal as possible, and $r-s$ parts with total size $n-m$, as equal as possible.
\end{theorem}

\begin{remark}
The finite term $\ex(q,F_1+K_1)$ is needed only when the complementary set has bounded size. Since $F_1+K_1$ is edge-critical with chromatic number $r-s+1$, Simonovits' color-critical Tur\'an theorem gives
\[
   \ex(q,F_1+K_1)=t_{r-s}(q)
\]
for all sufficiently large $q$. Thus Theorem~\ref{thm:2} gives the stated exact Tur\'an-type expression whenever $n-m$ lies in this Simonovits range.
\end{remark}

For an integer $t\ge 0$, we say that a graph $H$ is
\emph{$t$-progressive-edge-critical} if $H$ is edge-critical and there
exist vertices $v_1,\ldots,v_t$ such that, for each $1\le i\le t$, the graph
\[
   H-\{v_1,\ldots,v_i\}
\]
is edge-critical and has chromatic number $\chi(H)-i$. We call
\[
   H-\{v_1,\ldots,v_t\}
\]
a \emph{core} of $H$. For example, $K_t+C_{2k+1}$ is
$t$-progressive-edge-critical, with core $C_{2k+1}$.

When $t\ge 2$, every $t$-progressive-edge-critical graph is
double-edge-critical. In particular, if $H$ is $(r-s)$-progressive-edge-critical
with core $F$, then
\[
   H\subseteq K_{r-s}+F.
\]
In this setting, the auxiliary graph $J_0+K_1$ appearing in
Theorem~\ref{thm:2} is simply $K_{r-s+1}$. Hence Theorem~\ref{thm:2}
implies the following corollary. The same conclusion also holds in the
case $r=s+1$; we omit the minor modifications needed for this boundary case.

\begin{corollary}\label{cor:progressive}
Let $r>s\ge 1$, let $H$ be an $(r-s)$-progressive-edge-critical graph with
$\chi(H)=r+1$, and let $F$ be a core of $H$. Then there exists
$n_0=n_0(H,r,s)$ such that the following holds for all $n\ge n_0$.

Let $G$ be an $H$-free graph on $n$ vertices. Suppose that $G$ contains a
set $M$ of size $m$ such that $G[M]$ is $F$-free and
\[
   m\ge \left\lceil \frac{sn}{r}\right\rceil.
\]
Then
\[
   e(G)\le m(n-m)+t_s(m)+t_{r-s}(n-m).
\]
Moreover, equality holds if and only if $G$ is a complete $r$-partite graph
whose parts can be divided into two groups: $s$ parts with total size $m$,
as equal as possible, and $r-s$ parts with total size $n-m$, as equal as
possible.
\end{corollary}

The maximum-degree condition $\Delta(G)\ge m$ implies a special case of our
sparse-set condition. If $v$ is a vertex of degree at least $m$ and
$M\subseteq N(v)$ has size $m$, then $G[M]$ is $H-x$-free for every
$x\in V(H)$; otherwise, $v$ together with a copy of $H-x$ in $M$ would give a
copy of $H$. Thus our framework extends the maximum-degree setting.

Our methods also extend asymptotically to general forbidden graphs. In this setting, the large prescribed set is not required to be $K_{s+1}$-free; instead, it is enough to assume that it spans few copies of a fixed graph $F$ with $\chi(F)<\chi(H)$.

In the next theorem, we assume that the constant $c$ is strictly larger than
\[
   \frac{\chi(F)-1}{\chi(H)-1}.
\]
The auxiliary constants in the statement depend on the size of the gap
\[
   c-\frac{\chi(F)-1}{\chi(H)-1}.
\]
In particular, as this gap tends to zero, these constants may deteriorate rapidly.

\begin{theorem}\label{thm:3}
Let $H$ and $F$ be two fixed graphs with $\chi(H)>\chi(F)\ge 2$, let
\[
   c>\frac{\chi(F)-1}{\chi(H)-1},
\]
and let $\varepsilon>0$. Then there exist constants $\eta>0$ and $n_0$ such that the following holds for all $n\ge n_0$.

Suppose that $G$ is an $H$-free graph on $n$ vertices and contains a vertex subset $M$ of size $m$ with
\[
   \frac{m}{n}\ge c
\]
such that $G[M]$ contains at most
\[
   \eta m^{v(F)}
\]
copies of $F$. Then
\[
   e(G)
   \le
   m(n-m)
   +\left(1-\frac{1}{\chi(F)-1}\right)\binom{m}{2}
   +\left(1-\frac{1}{\chi(H)-\chi(F)}\right)\binom{n-m}{2}
   +\varepsilon n^2.
\]
\end{theorem}

\begin{remark}
The upper bound in Theorem~\ref{thm:3} is asymptotically best possible. Let
\[
   r=\chi(H)-1,
   \qquad
   p=\chi(F)-1.
\]
Consider the complete $r$-partite graph whose vertex classes can be divided into two groups: $p$ parts with total size $m$, as equal as possible, and $r-p$ parts with total size $n-m$, as equal as possible. Let $M$ be the union of the first $p$ parts. Then $G[M]$ is $p$-partite and hence $F$-free, while $G$ is $r$-partite and hence $H$-free. Moreover,
\[
   e(G)=m(n-m)+t_p(m)+t_{r-p}(n-m),
\]
which agrees with the leading term in Theorem~\ref{thm:3}, up to an $O(n)$ error.
\end{remark}

Several related questions have been investigated recently. See~\cite{LRW} for the effect of a large-degree vertex on the Andr\'asfai--Erd\H{o}s--S\'os theorem, and see~\cite{HHLLYZ} for analogous questions on Tur\'an problems for degenerate hypergraphs under a large-degree condition.

The precise formulation of Theorem~\ref{thm:1} enables a simple inductive proof. Theorem~\ref{thm:2} is proved by combining the same induction with Simonovits' color-critical Tur\'an theorem and a bounded-complement argument. Theorem~\ref{thm:3} then follows from Theorem~\ref{thm:1} via Szemer\'edi's Regularity Lemma and the Graph Removal Lemma.

The remainder of the paper is organized as follows. In Section~2, we collect the necessary notation and preliminary results. Section~3 is devoted to the proofs of Theorems~\ref{thm:1} and~\ref{thm:2}, while Section~4 contains the proof of Theorem~\ref{thm:3}. We conclude in Section~5 with several open problems.

\section{Notation and Preliminary Results}

Throughout the paper, all graphs are finite, simple, and undirected. For a graph $G$, we write $V(G)$ and $E(G)$ for its vertex set and edge set, respectively, and set
\[
   v(G)=|V(G)|,
   \qquad
   e(G)=|E(G)|.
\]

For a vertex $v\in V(G)$, the \emph{neighborhood} of $v$ is
\[
   N_G(v)=\{u\in V(G): uv\in E(G)\},
\]
and the \emph{degree} of $v$ is
\[
   d_G(v)=|N_G(v)|.
\]
When the graph is clear from context, we omit the subscript and simply write $N(v)$ and $d(v)$.

For a subset $X\subseteq V(G)$, let $G[X]$ denote the subgraph of $G$ induced by $X$. We write
\[
   e(X)=e(G[X])
\]
for the number of edges induced by $X$. For $s\ge 1$, we define
\[
   \alpha_s(G)
   =
   \max\bigl\{|M|: M\subseteq V(G) \text{ and } G[M] \text{ is } K_{s+1}\text{-free}\bigr\}.
\]
In particular, $\alpha_1(G)=\alpha(G)$ is the independence number of $G$.

For two disjoint subsets $X,Y\subseteq V(G)$, the \emph{induced bipartite subgraph} between $X$ and $Y$, denoted by $G[X,Y]$, is the bipartite graph with vertex classes $X$ and $Y$ and edge set
\[
   E(G[X,Y])
   =
   \{uv\in E(G): u\in X,\ v\in Y\}.
\]
We write
\[
   e(X,Y)=|E(G[X,Y])|.
\]

A graph is called \emph{complete $r$-partite} if its vertex set can be partitioned into $r$ independent sets
\[
   V(G)=V_1\cup\cdots\cup V_r
\]
such that every two vertices from distinct parts are adjacent. We allow parts of size zero when convenient.

For two vertex-disjoint graphs $A$ and $B$, their \emph{join}, denoted $A+B$, is obtained from their disjoint union by adding all edges between $V(A)$ and $V(B)$.

The \emph{Tur\'an graph} $T_r(n)$ is the complete $r$-partite graph on $n$ vertices whose vertex classes form an equitable partition; that is, the sizes of any two parts differ by at most one. Let
\[
   t_r(n)=e(T_r(n))
\]
denote the number of edges in $T_r(n)$. We use the convention $t_r(0)=0$.

For a graph $J$ and a positive integer $q$, let $J^q$ denote the $q$-blow-up of $J$: each vertex of $J$ is replaced by an independent set of size $q$, and each edge of $J$ is replaced by a complete bipartite graph between the corresponding independent sets. In particular, $K_s^q$ denotes the complete $s$-partite graph with all parts of size $q$.

We will use the following standard recurrence for Tur\'an numbers.

\begin{lemma}\label{lem:turan-difference}
Let $r\ge 1$, and write $n=qr+a$, where $0\le a<r$. Then
\[
   t_r(n+1)-t_r(n)=n-\left\lfloor \frac{n}{r}\right\rfloor.
\]
\end{lemma}

\begin{proof}
The vertex classes of $T_r(n)$ have sizes
\[
   q+1,\ldots,q+1,q,\ldots,q,
\]
with exactly $a$ classes of size $q+1$. To obtain $T_r(n+1)$, add one vertex to a smallest class, which has size $q$. The new vertex is adjacent to every vertex outside that class and therefore contributes
\[
   n-q
   =
   n-\left\lfloor\frac{n}{r}\right\rfloor
\]
new edges. This proves the recurrence.
\end{proof}

For integers $r>s\ge 1$, define
\begin{equation}\label{eq:tau}
   \tau_{r,s}(n,m)=m(n-m)+t_s(m)+t_{r-s}(n-m).
\end{equation}

We will use the following asymptotic form of this quantity.

\begin{lemma}\label{lem:limit}
Let $r>s\ge 1$, and let $m=m(n)$ satisfy $m/n\to c\in[0,1]$. Then
\[
   \lim_{n\to\infty}\frac{\tau_{r,s}(n,m)}{\binom{n}{2}}
   =
   1-\frac{c^2}{s}-\frac{(1-c)^2}{r-s}.
\]
\end{lemma}

\begin{proof}
Since $m/n\to c$, we have
\begin{align*}
   \frac{\tau_{r,s}(n,m)}{\binom{n}{2}}
   &=
   \frac{2m(n-m)}{n^2}
   +\left(1-\frac{1}{s}\right)\frac{m^2}{n^2}
   +\left(1-\frac{1}{r-s}\right)\frac{(n-m)^2}{n^2}
   +o(1) \\
   &=
   2c(1-c)
   +\left(1-\frac{1}{s}\right)c^2
   +\left(1-\frac{1}{r-s}\right)(1-c)^2
   +o(1) \\
   &=
   1-\frac{c^2}{s}-\frac{(1-c)^2}{r-s}+o(1).
\end{align*}
The claimed limit follows.
\end{proof}

\begin{lemma}\label{lem:low-degree}
Let $s\ge 1$, and let $G$ be a $K_{s+1}$-free graph on $N\ge 1$ vertices. Then $G$ has a vertex of degree at most
\[
   t_s(N)-t_s(N-1).
\]
\end{lemma}

\begin{proof}
By Lemma~\ref{lem:turan-difference},
\[
   t_s(N)-t_s(N-1)
   =
   N-1-\left\lfloor\frac{N-1}{s}\right\rfloor.
\]
Suppose, for a contradiction, that every vertex of $G$ has degree at least $t_s(N)-t_s(N-1)+1$. Then
\[
   e(G)\geq \frac{N}{2}\left(N-\left\lfloor\frac{N-1}{s}\right\rfloor\right).
\]
It remains to note that the expression on the right is larger than $t_s(N)$.

Write $N=sq+a$, where $0\le a<s$. If $a=0$, then
\begin{align*}
    &\frac{N}{2}\left(N-\left\lfloor\frac{N-1}{s}\right\rfloor\right)-t_s(N) \\
    &\qquad = \frac{sq}{2}(sq-q+1)-q^2\binom{s}{2}
    =\frac{sq}{2}>0.
\end{align*}
If $a>0$, then
\begin{align*}
    &\frac{N}{2}\left(N-\left\lfloor\frac{N-1}{s}\right\rfloor\right)-t_s(N) \\
    &\qquad = \frac{sq+a}{2}(sq+a-q)
    -\left(\binom{sq+a}{2}-(s-a)\binom{q}{2}-a\binom{q+1}{2}\right) \\
    &\qquad =\frac{a(q+1)}{2}>0.
\end{align*}
Thus $e(G)>t_s(N)$, contradicting Tur\'an's theorem. Hence some vertex of $G$ has degree at most $t_s(N)-t_s(N-1)$.
\end{proof}

\begin{lemma}\label{lem:threshold-turan}
Let $r>s\ge 1$, and let $N\ge 1$. If
\[
   m\in\left\{\left\lfloor\frac{sN}{r}\right\rfloor,\left\lceil\frac{sN}{r}\right\rceil\right\},
\]
then
\[
   \tau_{r,s}(N,m)=t_r(N).
\]
\end{lemma}

\begin{proof}
Write $N=qr+a$, where $0\le a<r$. In $T_r(N)$, there are $a$ parts of size $q+1$ and $r-a$ parts of size $q$. For
\[
   j\in\left\{\left\lfloor\frac{sa}{r}\right\rfloor,\left\lceil\frac{sa}{r}\right\rceil\right\},
\]
we have $0\le j\le \min\{s,a\}$. Thus we may choose $j$ of the parts of size $q+1$ and $s-j$ of the parts of size $q$. Their total size is $sq+j$, which is either $\lfloor sN/r\rfloor$ or $\lceil sN/r\rceil$. The edges inside these $s$ selected parts form $T_s(m)$, the edges inside the remaining $r-s$ parts form $T_{r-s}(N-m)$, and all cross-edges between the two groups are present. Hence
\[
   t_r(N)=m(N-m)+t_s(m)+t_{r-s}(N-m)=\tau_{r,s}(N,m).
\]
\end{proof}

\begin{theorem}[Simonovits' color-critical Tur\'an theorem~\cite{Simonovits1968}]\label{thm:simonovits}
Let $F$ be a color-critical graph with $\chi(F)=r+1$. Then there exists $n_0=n_0(F)$ such that, for every $n\ge n_0$, every $F$-free graph $G$ on $n$ vertices satisfies
\[
   e(G)\le t_r(n).
\]
Moreover, equality holds if and only if $G\cong T_r(n)$. Equivalently,
\[
   \operatorname{ex}(n,F)=t_r(n)
\]
for all sufficiently large $n$, and the unique extremal graph is the Tur\'an graph $T_r(n)$.
\end{theorem}

\begin{lemma}\label{lem:double-critical-join}
Let $r>s\ge 2$ with $r-s\ge 2$, and let $H$ be a double-edge-critical graph with $\chi(H)=r+1$. Then there exist edge-critical graphs $F_1$ and $F$ such that
\[
   H\subseteq F_1+F,
   \qquad
   \chi(F_1)=r-s,
   \qquad
   \chi(F)=s+1.
\]
Moreover, $F_1$ and $F$ may be chosen as induced subgraphs of $H$ on two complementary unions of color classes in a suitable proper $(r+1)$-coloring of $H$.
\end{lemma}

\begin{proof}
Let $e_1=x_1y_1$ and $e_2=x_2y_2$ be edges of $H$ such that
\[
   \chi\bigl(H-\{e_1,e_2\}\bigr)=r-1.
\]
By the preceding observation, $e_1$ and $e_2$ are vertex-disjoint. Fix a proper $(r-1)$-coloring of $H-\{e_1,e_2\}$, and write its color classes as
\[
   C_1,C_2,\ldots,C_{r-1}.
\]
In this coloring, the endpoints of each deleted edge must receive the same color. Indeed, if neither edge were monochromatic, the same coloring would be a proper coloring of $H$; if exactly one edge were monochromatic, then splitting one endpoint of that edge into a new color class would give a proper $r$-coloring of $H$. Both alternatives contradict $\chi(H)=r+1$.

The two deleted edges must be monochromatic in distinct color classes. If both pairs of endpoints lay in the same color class, then, in $H$, that class would induce only the two disjoint edges $e_1$ and $e_2$ and hence would be bipartite. Coloring this class with two colors and keeping the other $r-2$ color classes unchanged would give a proper $r$-coloring of $H$, again a contradiction. Thus, after relabeling, we may assume that $x_1,y_1\in C_1$ and $x_2,y_2\in C_2$.

Now split $x_1$ and $x_2$ into singleton color classes. Equivalently, replace $C_1$ and $C_2$ by
\[
   \{x_1\},\quad C_1\setminus\{x_1\},\quad \{x_2\},\quad C_2\setminus\{x_2\},
\]
and keep $C_3,\ldots,C_{r-1}$ unchanged. This gives a proper $(r+1)$-coloring of $H$ with two singleton color classes.

Since $r-s\ge 2$, choose one group of $r-s$ color classes consisting of $\{x_1\}$, $C_1\setminus\{x_1\}$, and any $r-s-2$ of the remaining old color classes different from $C_2$. Let $A$ be the union of these classes and put $B=V(H)\setminus A$. Then $B$ is the union of the remaining $s+1$ color classes and contains both endpoints of $e_2$.

Set
\[
   F_1=H[A],
   \qquad
   F=H[B].
\]
Since the join $F_1+F$ contains all possible edges between $A$ and $B$, we have $H\subseteq F_1+F$. The displayed coloring gives
\[
   \chi(F_1)\le r-s,
   \qquad
   \chi(F)\le s+1.
\]
Both inequalities are tight. If $\chi(F_1)\le r-s-1$, then coloring $F_1$ with at most $r-s-1$ colors and coloring $F$ by its displayed $s+1$ color classes would give a proper $r$-coloring of $H$. Similarly, if $\chi(F)\le s$, then coloring $F_1$ by its displayed $r-s$ color classes and coloring $F$ with at most $s$ colors would again give a proper $r$-coloring of $H$. Hence
\[
   \chi(F_1)=r-s,
   \qquad
   \chi(F)=s+1.
\]

Finally, $F_1$ is edge-critical. After deleting $e_1$, the singleton class $\{x_1\}$ can be merged with $C_1\setminus\{x_1\}$, so
\[
   \chi(F_1-e_1)\le r-s-1<\chi(F_1).
\]
Thus $e_1$ is a critical edge of $F_1$. The same argument with $e_2$ shows that $F$ is edge-critical. This completes the proof.
\end{proof}

\begin{lemma}\label{lem:critical-low-degree}
Let $F$ be a color-critical graph with $\chi(F)=s+1$. Then there exists $N_F$ such that every $F$-free graph $G$ on $N\ge N_F$ vertices contains a vertex of degree at most
\[
   t_s(N)-t_s(N-1).
\]
\end{lemma}

\begin{proof}
Choose $N_F$ so that Theorem~\ref{thm:simonovits} holds for $F$ for every $N\ge N_F$. If every vertex of $G$ had degree at least $t_s(N)-t_s(N-1)+1$, then the calculation in the proof of Lemma~\ref{lem:low-degree} would give $e(G)>t_s(N)$, contradicting Theorem~\ref{thm:simonovits}.
\end{proof}

\section{Proofs of Theorems~\ref{thm:1} and~\ref{thm:2}}

\begin{proof}[Proof of Theorem~\ref{thm:1}]
We prove the following slightly stronger fixed-set statement: if $M\subseteq V(G)$ has size $m$, $G[M]$ is $K_{s+1}$-free, and $m\ge \lceil sn/r\rceil$, then
\[
   e(G)\le \tau_{r,s}(n,m).
\]
The theorem follows by choosing such a set $M$ from the assumption $\alpha_s(G)\ge m$.

We argue by induction on $n$. The case $1\le n\le r$ is immediate. Assume $n\ge r+1$, and put
\[
   X=V(G)\setminus M.
\]
By Lemma~\ref{lem:low-degree}, the $K_{s+1}$-free graph $G[M]$ has a vertex $v\in M$ satisfying
\[
   d_{G[M]}(v)\le t_s(m)-t_s(m-1).
\]
Since $v$ has at most $n-m$ neighbors outside $M$, we have
\begin{equation}\label{eq:degree-bound-thm1}
   d_G(v)
   \le
   (n-m)+t_s(m)-t_s(m-1).
\end{equation}
Let
\[
   G'=G-v,
   \qquad
   M'=M\setminus\{v\}.
\]
Then $G'$ is $K_{r+1}$-free and $G'[M']$ is $K_{s+1}$-free.

If
\[
   m-1\ge \left\lceil\frac{s(n-1)}{r}\right\rceil,
\]
then the induction hypothesis gives
\[
   e(G')\le \tau_{r,s}(n-1,m-1).
\]
If instead
\[
   m-1< \left\lceil\frac{s(n-1)}{r}\right\rceil,
\]
then, since $m\ge \lceil sn/r\rceil$, we must have
\[
   m-1=\left\lfloor\frac{s(n-1)}{r}\right\rfloor.
\]
By Lemma~\ref{lem:threshold-turan},
\[
   \tau_{r,s}(n-1,m-1)=t_r(n-1).
\]
Thus Tur\'an's theorem again gives
\[
   e(G')\le \tau_{r,s}(n-1,m-1).
\]
Therefore, in all cases,
\begin{align*}
   e(G)
   &=e(G')+d_G(v) \\
   &\le \tau_{r,s}(n-1,m-1)+(n-m)+t_s(m)-t_s(m-1) \\
   &=(m-1)(n-m)+t_s(m-1)+t_{r-s}(n-m) \\
   &\qquad +(n-m)+t_s(m)-t_s(m-1) \\
   &=m(n-m)+t_s(m)+t_{r-s}(n-m) \\
   &=\tau_{r,s}(n,m).
\end{align*}
This proves the upper bound.

It remains to discuss equality. The construction in the statement is clearly $K_{r+1}$-free. If $M$ is the union of the first $s$ parts, then $G[M]\cong T_s(m)$, so $G[M]$ is $K_{s+1}$-free, and
\[
   e(G)=m(n-m)+t_s(m)+t_{r-s}(n-m).
\]
Thus equality is attained.

Conversely, suppose equality holds. Then equality must hold throughout the induction argument above. In particular, the chosen vertex $v\in M$ satisfies
\begin{equation}\label{eq:cross-complete-thm1}
   d_G(v,V(G)\setminus M)=n-m
\end{equation}
and
\begin{equation}\label{eq:internal-degree-thm1}
   d_{G[M]}(v)=t_s(m)-t_s(m-1)
   =m-1-\left\lfloor\frac{m-1}{s}\right\rfloor.
\end{equation}
If the induction hypothesis applies to $G'$, then equality also holds for $G'$. By induction, $G'$ is a complete $r$-partite graph whose parts can be divided into $s$ parts with total size $m-1$, as equal as possible, and $r-s$ parts with total size $n-m$, as equal as possible. Equation~\eqref{eq:cross-complete-thm1} says that $v$ is adjacent to every vertex outside $M$.

Inside $M'$, the graph is $T_s(m-1)$. Since $G[M]$ is $K_{s+1}$-free, the neighborhood of $v$ in $M'$ cannot meet all $s$ parts of this $T_s(m-1)$. Hence $v$ is nonadjacent to every vertex in at least one part of $T_s(m-1)$. By~\eqref{eq:internal-degree-thm1}, the number of non-neighbors of $v$ in $M'$ is exactly
\[
   \left\lfloor\frac{m-1}{s}\right\rfloor,
\]
which is the size of a smallest part of $T_s(m-1)$. Therefore $v$ is nonadjacent precisely to one smallest part of $T_s(m-1)$ and is adjacent to every other vertex of $M'$. Adding $v$ to that smallest part gives $T_s(m)$ on $M$. Together with~\eqref{eq:cross-complete-thm1}, this gives the complete $r$-partite graph described in the theorem.

It remains only to consider the boundary step, where
\[
   m-1=\left\lfloor\frac{s(n-1)}{r}\right\rfloor
\]
and the proof used Tur\'an's theorem for $G'$. Equality then implies
\[
   G'\cong T_r(n-1).
\]
Since $G$ is $K_{r+1}$-free, the neighborhood of $v$ in $G'$ is $K_r$-free. In the complete $r$-partite graph $T_r(n-1)$, this means that the neighborhood of $v$ must miss an entire part. By~\eqref{eq:cross-complete-thm1}, $v$ is adjacent to every vertex outside $M$, so the missed part lies inside $M'$. The number of non-neighbors of $v$ in $M'$ is, by~\eqref{eq:internal-degree-thm1},
\[
   \left\lfloor\frac{m-1}{s}\right\rfloor.
\]
Writing $n-1=qr+a$, one has
\[
   m-1=\left\lfloor\frac{s(n-1)}{r}\right\rfloor=sq+\left\lfloor\frac{sa}{r}\right\rfloor,
\]
and hence
\[
   \left\lfloor\frac{m-1}{s}\right\rfloor=q.
\]
Thus the missed part is a smallest part of $T_r(n-1)$, and $v$ is adjacent to every vertex outside that part. Adding $v$ to this smallest part gives $T_r(n)$, which admits the required division of its parts into $s$ parts with total size $m$ and $r-s$ parts with total size $n-m$. This completes the equality characterization.
\end{proof}

\begin{proof}[Proof of Theorem~\ref{thm:2}]
Put
\[
   t=r-s,
   \qquad
   J=F_1+F,
   \qquad
   J_0=F_1+K_1.
\]
Since $F_1$ and $F$ are edge-critical, both $J$ and $J_0$ are edge-critical, with
\[
   \chi(J)=r+1,
   \qquad
   \chi(J_0)=t+1.
\]
Moreover, $H\subseteq J$, so every $H$-free graph is also $J$-free. Choose a vertex $u\in V(F)$ incident with a critical edge of $F$. Then
\[
   \chi(F-u)\le s.
\]

Choose $q_0$ sufficiently large so that Simonovits' color-critical Tur\'an theorem applies to $J_0$ for every order at least $q_0$. Thus
\[
   \ex(q,J_0)=t_t(q)
   \qquad\text{for all }q\ge q_0.
\]
Increase $q_0$, if necessary, so that whenever $q\ge q_0$ and
\[
   m\ge \left\lceil\frac{s(m+q)}{r}\right\rceil,
\]
we have $m$ large enough for Lemma~\ref{lem:critical-low-degree} to apply to $F$. Finally choose $n_0$ large enough so that Simonovits' theorem applies to $J$ on all orders arising in the boundary step below, and so that the bounded-complement argument at the end applies whenever $n\ge n_0$.

We first prove the exact bound in the range $q=n-m\ge q_0$. The proof is the same induction as in Theorem~\ref{thm:1}, with Simonovits' theorem replacing Tur\'an's theorem. Since $G[M]$ is $F$-free and $m$ is sufficiently large, Lemma~\ref{lem:critical-low-degree} gives a vertex $v\in M$ such that
\[
   d_{G[M]}(v)\le t_s(m)-t_s(m-1).
\]
Hence
\[
   d_G(v)\le q+t_s(m)-t_s(m-1).
\]
Let
\[
   G'=G-v,
   \qquad
   M'=M\setminus\{v\}.
\]
If
\[
   m-1\ge \left\lceil\frac{s(n-1)}{r}\right\rceil,
\]
then the induction hypothesis applies to $G'$ and $M'$, since the value of $q=n-m$ is unchanged. If this inequality fails, then
\[
   m-1=\left\lfloor\frac{s(n-1)}{r}\right\rfloor.
\]
Since $G'$ is $J$-free and has sufficiently large order, Simonovits' theorem gives
\[
   e(G')\le t_r(n-1).
\]
By Lemma~\ref{lem:threshold-turan},
\[
   t_r(n-1)=\tau_{r,s}(n-1,m-1).
\]
Thus in all cases
\[
   e(G')\le \tau_{r,s}(n-1,m-1).
\]
Consequently,
\begin{align*}
   e(G)
   &=e(G')+d_G(v) \\
   &\le \tau_{r,s}(n-1,m-1)+q+t_s(m)-t_s(m-1) \\
   &=m q+t_s(m)+t_t(q) \\
   &=\tau_{r,s}(n,m).
\end{align*}
This proves the desired exact bound when $q\ge q_0$.

It remains to prove the finite-complement estimate for $q<q_0$. Since $q$ is bounded and $n$ is sufficiently large, $m$ is sufficiently large for Simonovits' theorem to apply to $F$, and hence
\[
   e(G[M])\le t_s(m).
\]
Suppose, for a contradiction, that
\[
   e(G)>mq+t_s(m)+\ex(q,J_0).
\]
Let $X=V(G)\setminus M$, and define
\[
   A=t_s(m)-e(G[M]),
   \qquad
   B=mq-e(M,X),
   \qquad
   C=e(G[X])-\ex(q,J_0).
\]
Then $A\ge0$, $B\ge0$, and the assumed strict inequality gives
\[
   C>A+B.
\]
In particular, $C>0$, so $G[X]$ contains a copy $Y$ of $J_0=F_1+K_1$. Since
\[
   B<C\le \binom q2\le \binom{q_0}{2},
\]
the common neighborhood of all vertices of $Y$ inside $M$ has size at least
\[
   m-B\ge m-\binom{q_0}{2}.
\]
Denote this common neighborhood by $U$.

Also,
\[
   e(G[M])=t_s(m)-A\ge t_s(m)-\binom{q_0}{2}.
\]
Deleting the at most $B$ vertices of $M\setminus U$ removes at most $Bm$ edges, so
\[
   e(G[U])
   \ge t_s(m)-O_{F_1,F,r,s}(m).
\]
Since $|U|=m-O_{F_1,F,r,s}(1)$ and $\chi(F-u)\le s$, the set $U$ contains a copy of $F-u$ for all sufficiently large $n$. Indeed, if $F-u$ is edgeless, this is immediate from $|U|\to\infty$. Otherwise, the Erd\H{o}s--Stone--Simonovits theorem gives
\[
   \ex(|U|,F-u)
   \le
   \left(1-\frac{1}{s-1}+o(1)\right)\binom{|U|}{2},
\]
where the right-hand side is interpreted as $o(|U|^2)$ when $s=2$. This is smaller than $e(G[U])\ge t_s(m)-O_{F_1,F,r,s}(m)$ for all sufficiently large $n$.

In the copy $Y\cong F_1+K_1$, use the $K_1$-vertex to play the vertex $u$. Since every vertex of $Y$ is adjacent to every vertex of $U$, the copy of $F-u$ in $U$ together with $Y$ contains a copy of $F_1+F$, and hence contains a copy of $H$. This contradicts the assumption that $G$ is $H$-free. Therefore
\[
   e(G)\le mq+t_s(m)+\ex(q,J_0)
\]
for $q<q_0$ as well.

The equality statement in the range $q\ge q_0$ follows from the equality cases in the induction. Equality forces equality in Lemma~\ref{lem:critical-low-degree} and in Simonovits' theorem for $F$ and $J$. Hence $G[M]\cong T_s(m)$, $G[X]\cong T_t(q)$, and all edges between $M$ and $X$ are present. The same part-by-part argument used in the proof of Theorem~\ref{thm:1} then shows that $G$ is the complete $r$-partite graph described in the statement. Conversely, that graph is $r$-partite, and hence $H$-free because $\chi(H)=r+1$, and it has exactly
\[
   m(n-m)+t_s(m)+t_{r-s}(n-m)
\]
edges.
\end{proof}

\section{Proof of Theorem~\ref{thm:3}}

The main tools are the graph removal lemma~\cite{ruzsa1976triple} and Szemer\'edi's regularity lemma~\cite{szemeredi1975}. We use the following standard terminology, following Diestel~\cite[Section~7.4]{diestel2024}.

Let $G=(V,E)$ be a graph. For two disjoint nonempty sets $X,Y\subseteq V$, define the density between $X$ and $Y$ by
\[
   d(X,Y)=\frac{e(X,Y)}{|X||Y|}.
\]
For $\epsilon>0$, a pair of disjoint sets $A,B\subseteq V$ is called \emph{$\epsilon$-regular} if every $X\subseteq A$ and $Y\subseteq B$ satisfying
\[
   |X|\ge \epsilon |A|,
   \qquad
   |Y|\ge \epsilon |B|
\]
also satisfies
\[
   |d(X,Y)-d(A,B)|\le \epsilon.
\]

A partition $\{V_0,V_1,\ldots,V_k\}$ of $V$ is called an \emph{$\epsilon$-regular partition} of $G$, with exceptional set $V_0$, if
\begin{enumerate}
   \item $|V_0|\le \epsilon |V|$;
   \item $|V_1|=\cdots=|V_k|$;
   \item all but at most $\epsilon k^2$ of the pairs $(V_i,V_j)$ with $1\le i<j\le k$ are $\epsilon$-regular.
\end{enumerate}

Given an $\epsilon$-regular partition $\{V_0,V_1,\ldots,V_k\}$ with
\[
   |V_1|=\cdots=|V_k|=\ell
\]
and a density parameter $d\in(0,1]$, the \emph{regularity graph} $R$ has vertex set
\[
   V(R)=\{V_1,\ldots,V_k\},
\]
where $V_iV_j\in E(R)$ if and only if $(V_i,V_j)$ is $\epsilon$-regular in $G$ and
\[
   d(V_i,V_j)\ge d.
\]

\begin{lemma}[Regularity lemma with an initial partition; see~\cite{szemeredi1975}]\label{lem:reg}
For every $\epsilon>0$ and all integers $a,t\ge 1$, there exists an integer $b$ such that the following holds. For every graph $G$ and every partition $\mathcal P$ of $V(G)$ into at most $t$ parts, the graph $G$ admits an $\epsilon$-regular partition
\[
   \{V_0,V_1,\ldots,V_k\}
\]
with
\[
   a\le k\le b,
\]
such that each $V_i$, $i\ge 1$, is contained in some member of $\mathcal P$.
\end{lemma}

\begin{lemma}[Embedding lemma; see~\cite{szemeredi1975}]\label{lem:red}
For every $d\in(0,1]$ and every integer $\Delta\ge 1$, there exists $\epsilon_0>0$ with the following property. Let $G$ be a graph, let $H$ be a graph with $\Delta(H)\le \Delta$, and let $q\in\mathbb N$. If $R$ is a regularity graph of $G$ with parameters $\epsilon,\ell,d$, where
\[
   \epsilon\le \epsilon_0
   \qquad\text{and}\qquad
   \ell\ge \frac{2q}{d^\Delta},
\]
then
\[
   H\subseteq R^q
   \qquad\Longrightarrow\qquad
   H\subseteq G.
\]
\end{lemma}

\begin{lemma}[Graph removal lemma; see~\cite{ruzsa1976triple}]\label{lem:removal}
For every fixed graph $F$ and every $\mu>0$, there exists $\eta>0$ such that every graph $G$ on $N$ vertices containing at most $\eta N^{v(F)}$ copies of $F$ can be made $F$-free by deleting at most $\mu N^2$ edges.
\end{lemma}

\begin{proof}[Proof of Theorem~\ref{thm:3}]
It is enough to prove the theorem for $0<\varepsilon<1$. Put
\[
   r=\chi(H)-1,
   \qquad
   p=\chi(F)-1.
\]
Then $r>p\ge 1$. Define
\[
   B(n,m)
   =
   m(n-m)
   +\left(1-\frac{1}{p}\right)\binom{m}{2}
   +\left(1-\frac{1}{r-p}\right)\binom{n-m}{2},
\]
and, for $x\in[0,1]$, define
\[
   \psi(x)=1-\frac{x^2}{p}-\frac{(1-x)^2}{r-p}.
\]
Here $\psi(x)$ is the coefficient of the main term of $B(n,m)$. For $x\ge p/r$, the function $\psi$ is nonincreasing, since
\[
   \psi'(x)=\frac{2(p-rx)}{p(r-p)}\le 0.
\]
Choose $\xi>0$ such that
\[
   0<\xi<\frac{1}{4}\left(c-\frac{p}{r}\right)
\]
and, for every $x\in[c,1]$,
\begin{equation}\label{eq:psi-continuity}
   \psi(x-\xi)\le \psi(x)+\frac{\varepsilon}{8}.
\end{equation}

Let
\[
   L=\max\{v(H),v(F)\},
   \qquad
   \Delta=\max\{\Delta(H),\Delta(F)\}.
\]
Choose a density parameter $d>0$ small enough that $d<\varepsilon/16$, and let $\epsilon_0$ be given by Lemma~\ref{lem:red} with parameters $d$ and $\Delta$. Choose $\epsilon>0$ such that
\begin{equation}\label{eq:epsilon-choice}
   \epsilon\le \epsilon_0,
   \qquad
   \epsilon<\frac{\xi}{4},
   \qquad
   4\epsilon+d<\frac{\varepsilon}{4}.
\end{equation}
Choose an integer $a$ sufficiently large that
\begin{equation}\label{eq:a-choice}
   \frac{1}{a}<\min\left\{\frac{\xi}{2},\frac{\varepsilon}{4}\right\}
\end{equation}
and such that, for every $k\ge a$ and every integer $q$ with $0\le q\le k$,
\begin{equation}\label{eq:uniform-tau}
   \frac{2\tau_{r,p}(k,q)}{k^2}
   \le
   \psi\left(\frac{q}{k}\right)+\frac{\varepsilon}{8}.
\end{equation}
This is possible by Lemma~\ref{lem:limit}, uniformly in $q$.

Apply Lemma~\ref{lem:reg} with parameters $\epsilon$, $a$, and $t=2$, and let $b$ be the resulting upper bound on the number of parts. Apply Lemma~\ref{lem:removal} to $F$ with
\[
   \mu=\frac{\varepsilon}{8},
\]
and let $\eta>0$ be the corresponding constant. Finally, choose $n_0$ sufficiently large that, for all $n\ge n_0$ and all $m\le n$,
\begin{equation}\label{eq:B-asymptotic}
   \frac{2B(n,m)}{n^2}
   \ge
   \psi\left(\frac{m}{n}\right)-\frac{\varepsilon}{8},
\end{equation}
and also
\begin{equation}\label{eq:cluster-size}
   \frac{(1-\epsilon)n}{b}\ge \frac{2L}{d^\Delta}.
\end{equation}

Suppose, for a contradiction, that $G$ satisfies the hypotheses but
\begin{equation}\label{eq:contradiction-assumption}
   e(G)>B(n,m)+\varepsilon n^2.
\end{equation}
By the choice of $\eta$ and the graph removal lemma, we may delete at most
\[
   \frac{\varepsilon}{8}m^2\le \frac{\varepsilon}{8}n^2
\]
edges from $G[M]$ and obtain a graph $G'$ such that $G'[M]$ is $F$-free. From~\eqref{eq:contradiction-assumption},
\begin{equation}\label{eq:G-prime-edge-lower}
   e(G')>B(n,m)+\frac{7\varepsilon}{8}n^2.
\end{equation}

Apply the regularity lemma with initial partition
\[
   \mathcal P=\{M,V(G)\setminus M\}
\]
to $G'$. We obtain an $\epsilon$-regular partition
\[
   \{V_0,V_1,\ldots,V_k\}
\]
with
\[
   a\le k\le b,
   \qquad
   |V_1|=\cdots=|V_k|=\ell,
\]
such that every $V_i$, $i\ge 1$, is contained either in $M$ or in $V(G)\setminus M$. By~\eqref{eq:cluster-size},
\[
   \ell=\frac{n-|V_0|}{k}
   \ge
   \frac{(1-\epsilon)n}{b}
   \ge
   \frac{2L}{d^\Delta}.
\]
Let $R$ be the regularity graph of $G'$ with parameters $\epsilon,\ell,d$.

Let
\[
   I=\{i\in[k]: V_i\subseteq M\}.
\]
Since all nonexceptional clusters refine the partition $\{M,V(G)\setminus M\}$,
\[
   |I|\ell\ge m-|V_0|\ge m-\epsilon n.
\]
As $\ell\le n/k$, this gives
\begin{equation}\label{eq:I-lower}
   |I|\ge \left(\frac{m}{n}-\epsilon\right)k.
\end{equation}

We claim that $R[I]$ is $K_{p+1}$-free. Indeed, if $R[I]$ contained a copy of $K_{p+1}$, then
\[
   F\subseteq K_{p+1}^L\subseteq R[I]^L.
\]
By Lemma~\ref{lem:red}, applied inside $G'[M]$, this would imply $F\subseteq G'[M]$, contradicting the fact that $G'[M]$ is $F$-free. Hence
\begin{equation}\label{eq:alpha-R}
   \alpha_p(R)\ge |I|.
\end{equation}
Similarly, $R$ is $K_{r+1}$-free. Otherwise, since
\[
   H\subseteq K_{r+1}^L,
\]
Lemma~\ref{lem:red} would imply $H\subseteq G'\subseteq G$, contradicting the assumption that $G$ is $H$-free.

We now estimate $e(R)$ from below. Every edge of $G'$ is either incident with $V_0$, lies inside one cluster, lies between an irregular pair, lies between a regular pair of density less than $d$, or corresponds to an edge of $R$. Therefore
\[
   e(G')
   \le
   e(R)\ell^2
   +|V_0|n
   +\epsilon k^2\ell^2
   +d\binom{k}{2}\ell^2
   +k\binom{\ell}{2}.
\]
Using $|V_0|\le \epsilon n$, $k\ell\le n$, and $k\ge a$, we obtain
\begin{equation}\label{eq:R-edge-lower}
   \frac{2e(R)}{k^2}
   \ge
   \frac{2e(G')}{n^2}-4\epsilon-d-\frac{1}{a}.
\end{equation}
Combining~\eqref{eq:B-asymptotic}, \eqref{eq:G-prime-edge-lower}, \eqref{eq:epsilon-choice}, \eqref{eq:a-choice}, and~\eqref{eq:R-edge-lower}, with $x=m/n$, yields
\begin{equation}\label{eq:R-too-many}
   \frac{2e(R)}{k^2}
   >
   \psi(x)+\varepsilon.
\end{equation}

Now set
\[
   q=\left\lceil (x-\xi)k\right\rceil.
\]
Since $x\ge c$ and $\xi<(c-p/r)/4$, we have
\[
   q\ge \left\lceil\frac{pk}{r}\right\rceil.
\]
Moreover, by~\eqref{eq:a-choice},
\[
   q\le (x-\xi)k+1\le \left(x-\frac{\xi}{2}\right)k,
\]
and by~\eqref{eq:epsilon-choice} and~\eqref{eq:I-lower}, this gives
\[
   q\le |I|.
\]
Thus $\alpha_p(R)\ge q$. Since $R$ is $K_{r+1}$-free, Theorem~\ref{thm:1} gives
\[
   e(R)
   \le
   q(k-q)+t_p(q)+t_{r-p}(k-q)
   =\tau_{r,p}(k,q).
\]
Using~\eqref{eq:uniform-tau}, \eqref{eq:psi-continuity}, and the fact that $\psi$ is nonincreasing on $[p/r,1]$, we get
\[
   \frac{2e(R)}{k^2}
   \le
   \frac{2\tau_{r,p}(k,q)}{k^2}
   \le
   \psi\left(\frac{q}{k}\right)+\frac{\varepsilon}{8}
   \le
   \psi(x-\xi)+\frac{\varepsilon}{8}
   \le
   \psi(x)+\frac{\varepsilon}{4},
\]
contradicting~\eqref{eq:R-too-many}. This contradiction proves Theorem~\ref{thm:3}.
\end{proof}

\section{Concluding Remarks}

Theorem~\ref{thm:1} gives an exact answer when the forbidden graph is a clique and the large prescribed set is $K_{s+1}$-free. Theorem~\ref{thm:2} extends this exact result to forbidden graphs that embed into the join of two edge-critical graphs, with a finite extremal correction for bounded complementary sets. Theorem~\ref{thm:3} gives an asymptotic analogue for general forbidden graphs $H$ and prescribed sets that span few copies of a fixed graph $F$. A natural next problem is to seek exact results, or sharp stability statements, in this more general setting.

\begin{problem}
Let $H$ and $F$ be fixed graphs with $\chi(H)>\chi(F)\ge 2$. Determine, for special families of $H$ and $F$, the maximum number of edges in an $H$-free graph on $n$ vertices containing a set $M$ of size $m$ such that $G[M]$ is $F$-free.
\end{problem}

Another natural direction is to characterize stability. Namely, if an $H$-free graph satisfying the hypotheses of Theorem~\ref{thm:3} has close to the maximum possible number of edges, must it be close to the complete multipartite construction described in the remark following the theorem?

\medskip
\noindent\textbf{Acknowledgements.} ChatGPT 5.5 are used to improve the presentation of the manuscript; 
the main concepts, ideas, and proofs were developed independently by the authors
and we take full responsibility for the contents of the paper.


\begin{thebibliography}{99}
\small

\bibitem{BBRS}
P. Balister, B. Bollob\'as, O. Riordan, and R. Schelp,
Graphs with large maximum degree containing no odd cycles of a given length,
\emph{Journal of Combinatorial Theory, Series B} \textbf{87} (2003), no.~2, 429--439.

\bibitem{BCL}
N. Behague, D. Chakraborti, and X. Liu,
Sabotaging Mantel's Theorem,
\emph{The Electronic Journal of Combinatorics} \textbf{33} (2026), no.~1.

\bibitem{diestel2024}
R.~Diestel,
\emph{Graph Theory}, 6th~ed.,
Graduate Texts in Mathematics, vol.~173,
Springer, Berlin, 2024.

\bibitem{HHLLYZ}
J. Hou, C. Hu, H. Li, X. Liu, C. Yang, and Y. Zhang,
On the boundedness of degenerate hypergraphs,
\emph{Acta Mathematica Sinica, English Series} \textbf{42} (2026), no.~2, 464--480.

\bibitem{HY}
Q. Huo and L. Yuan,
Graphs with large maximum degree containing no edge-critical graphs,
\emph{European Journal of Combinatorics} \textbf{106} (2022), 103576.

\bibitem{LRW}
X. Liu, S. Ren, and J. Wang,
Andr\'asfai--Erd\H{o}s--S\'os theorem under max-degree constraints,
arXiv:2512.10190 (2025).

\bibitem{ruzsa1976triple}
I.~Ruzsa and E.~Szemer\'edi,
Triple systems with no six points carrying three triangles,
\emph{Combinatorics (Proc. Fifth Hungarian Colloq., Keszthely, 1976)},
Colloq. Math. Soc. Janos Bolyai, vol.~18, North-Holland, Amsterdam, 1978, pp.~939--945.

\bibitem{Simonovits1968}
M.~Simonovits,
A method for solving extremal problems in graph theory, stability problems,
in \emph{Theory of Graphs (Proc. Colloq., Tihany, 1966)},
Academic Press, New York, 1968, pp.~279--319.

\bibitem{szemeredi1975}
E.~Szemer\'edi,
On sets of integers containing no $k$ elements in arithmetic progression,
\emph{Acta Arithmetica} \textbf{27} (1975), 199--245.

\end{thebibliography}
\end{document}